\documentclass[11pt,a4paper]{amsart}

\usepackage{amsfonts,amsmath,amssymb,amsthm}

\usepackage[utf8]{inputenc}
\usepackage[T1]{fontenc}

\newcommand{\Cc}{\mathbb{C}}

\newcommand{\Nn}{\mathbb{N}}
\newcommand{\Zz}{\mathbb{Z}}
\newcommand{\Qq}{\mathbb{Q}}
\newcommand{\Ff}{\mathbb{F}}

\renewcommand {\leq}{\leqslant}
\renewcommand {\geq}{\geqslant}

\renewcommand{\ast}{\star}

{\theoremstyle{plain}
\newtheorem{theorem}{Theorem}[section]    
\newtheorem{lemma}[theorem]{Lemma}       
\newtheorem{proposition}[theorem]{Proposition}      
\newtheorem{corollary}[theorem]{Corollary}      
\newtheorem*{conjecture*}{Conjecture}
\newtheorem*{theorem*}{Theorem}
}
{\theoremstyle{remark}

\newtheorem{definition}[theorem]{Definition}      
\newtheorem*{remark*}{Remark}  
\newtheorem{remark}[theorem]{Remark}   
\newtheorem{example}[theorem]{Example}
\newtheorem{questions}[theorem]{Problem}

}

\usepackage[a4paper]{geometry}
\title{The relative Schinzel hypothesis}

\author{Arnaud Bodin}
\author{Pierre D\`ebes}
\author{Salah Najib}

\email{arnaud.bodin@univ-lille.fr}
\email{pierre.debes@univ-lille.fr}
\email{slhnajib@gmail.com}
\address{Universit\'e de Lille, CNRS, Laboratoire Paul Painlev\'e, 59000 Lille, France}
\address{Universit\'e de Lille, CNRS, Laboratoire Paul Painlev\'e, 59000 Lille, France}
\address{Laboratoire ATRES, Facult\'e Polydisciplinaire de Khouribga, Universit\'e Sultan Moulay Slimane, BP 145, Hay Ezzaytoune, 25000 Khouribga, Maroc.}



\thanks{{\it Acknowledgment}. 
This work was supported in part by the Labex CEMPI  (ANR-11-LABX-0007-01) 
and by the ANR project ``LISA'' (ANR-17-CE40–0023-01). The first author thanks the University of British Columbia for his visit at PIMS}

\date{\today}

\begin{document}

\begin{abstract} 
The Schinzel Hypothesis is a conjecture about irreducible polynomials in one variable over the integers: under some standard condition, they should assume infinitely many prime values
at integers. We consider a relative version: if the polynomials are relati\-ve\-ly prime and no prime number divides all their values at integers, then they assume relatively prime values at at least one integer. We extend the question to all integral domains and prove it for a number of them: PIDs, UFDs containing an infinite field, polynomial rings over a UFD.
Applications include a new ``integral'' version of the Hilbert Irreducibility Theorem, for which the irreducibility conclusion is over the ring.
\end{abstract}

\maketitle


\section{Introduction} \label{sec:intro}

The so-called Schinzel Hypothesis (H) was stated in  \cite{schinzel-sierpinski}. Let $P_1, \ldots, P_s\in \Zz[y]$ be $s$ polynomials,
irreducible, of degree $\geq 1$ and satisfying this
{\it Assumption on Values}:
\vskip 1,5mm
 
 \noindent
 (AV1) {\it no prime number $p\in \Zz$ divides all integers $\prod_{i=1}^s P_i(m)$ with $m\in \Zz$.}
 \vskip 1,5mm
\noindent
Hypothesis (H) concludes that there are infinitely many $m\in \Zz$ 
such that $P_1(m),\ldots, P_s(m)$ are prime numbers. 
If true, the Schinzel Hypothesis would solve many classical problems in number theory:
the Twin Prime problem (take $(P_1,P_2)=(y,y+2)$), the infiniteness 
of primes of the form $m^2+1$ (take $P_1=y^2+1$), etc. 
However it is wide open except for one polynomial $P_1$ of degree one, in which case it
is the already profound Dirichlet theorem about primes in an arithmetic progression.
\vskip 0,5mm

We investigate the following {\it relative} version, which we raise for more 
general rings $Z$ than the ring $\Zz$ of integers. {\it Primes} of $Z$, \hbox{i.e.}
elements $p\in Z$ such that the ideal $pZ\subset Z$ is a prime ideal, replace prime numbers;
and we are interested in the weaker  {\it ``relatively prime''} conclusion that $P_1(m)$,\ldots,$P_s(m)$ 
have no common divisor in $Z$ (other than units)\footnote{From now on, we omit to say ``other than units'' 
 and consider that it is implicitly meant. We will also not use the somewhat ambiguous phrase ``relatively prime'', 
sometimes used for the stronger (though equivalent in PIDs) pro\-per\-ty ``generating the unit ideal''; we prefer 
to explicitly state the properties.}.

\begin{definition}
\label{th:schinzel-coprime}
Given an integral domain $Z$ of fraction field $Q$, say that the {\it relative Schinzel Hypothesis holds for $Z$} if the following is true. Let $P_1,\ldots,P_s \in Z[y]$ be $s\geq 2$ nonzero po\-ly\-no\-mials, with no common divisor in $Q[y]$ and 
satisfying this
  {\it Assumption on Values}: 
\vskip 1,5mm

\noindent
(AV2) {\it no prime $p$ of $Z$ divides all elements $P_1(m),\ldots, P_s(m)$ with $m\in Z$}.

\vskip 1,5mm

\noindent
Then there exists $m\in Z$ such that $P_1(m)$,\ldots,$P_s(m)$ have no common divisor in $Z$.
\end{definition}

All integral domains $Z$ which we know have the property are collected below. 

\begin{theorem} \label{thm:main_theorem} The relative Schinzel Hypothesis holds for $Z$ in the following cases:
 \vskip 0,3mm

\noindent
{\rm (a)} $Z$ is a Principal Ideal Domain {\rm (PID)}.
\vskip 0,5mm

\noindent
{\rm (b)}  $Z$ is a Unique Factorization Domain {\rm (UFD)} containing an infinite field.
\vskip 0,5mm

\noindent
{\rm (c)}  $Z$ is a polynomial ring $R[u_1,\ldots,u_r]$ $(r\geq 1)$ over a {\rm UFD} $R$. 
 \end{theorem}
 
Cases (a) and (b) are ``easy'' (see Section \ref{ssec:proofs_PID}); for $Z=\Zz$, a special case of the property with $s=2$  already appears in Schinzel's paper \cite[Lemme]{Sc59}.
Case (c) is more involved. The strategy rests on \cite{BDN19}, where the original Schinzel Hypothesis  itself  (suitably generalized) is shown for polynomial rings $R[u_1,\ldots,u_r]$ over a UFD $R$ with good ``Hilbertian properties''.  The Schinzel Hypothesis being stronger than the relative variant for these rings (Proposition \ref{prop:schinzel-implies-coprime}), the latter holds for them as well. Extending it to polynomial rings over an arbitrary UFD is the remaining challenge and the 
purpose of 
Section \ref{ssec:SH_more_general_rings}.

\begin{example}
The ring $\Zz$  is an example of (a) not satisfying (b) or (c). 
If $k$ is an infinite field, 
rings $k[[u_1,\ldots,u_r]]$ of formal power series in $r$ variables are examples of (b), but do not satisfy (a) or (c) if $r\geq 2$. 
More generally, regular local rings, which are UFDs by the Auslander-Buchsbaum 
theorem \cite{auslander-buchsbaum}, are examples of (b) if they contain an infinite field. Polynomial 
rings $\Zz[u_1,\ldots,u_r]$ are examples of (c) not satisfying (a) or (b). The trivial case that $Z$ is a field is contained 
in (a). 
\end{example}

In the three cases of Theorem \ref{thm:main_theorem}, the ring $Z$ is a UFD. It is the natural context for the relative Schinzel Hypothesis (gcd exist, primes are exactly the irreducible elements, Gauss's lemma is available, etc.).
See however Remark \ref{rem:final} for
a generalization of Theorem \ref{thm:main_theorem}(c) with $R$ not necessarily a UFD.

\begin{questions}
There remain rings of interest, even UFDs, for which the status of the relative Schinzel Hypothesis is unclear,  
e.g. $\Ff_q[[u_1,u_2]]$ and $\Zz_p[[u]]$. 
It would be nice to clarify these cases. More generally, finding an integral domain for which the relative Schinzel Hypothesis fails,
or, showing on the contrary that it always holds would be interesting. 
\end{questions}

\begin{remark}[{\rm on Assumption on Values}] \label{rem:AV2}
Assumption (AV2) is obviously necessary for the conclusion of the relative Schinzel Hypothesis to hold. (AV2) fails for example for $Z=\Zz$, $P_1(y) = y^p-y+p$, $P_2(y) = y^p-y$ with $p$ a prime number: all values $P_1(m), P_2(m)$ ($m\in \Zz$) are divisible by $p$.  (AV2) may also fail for polynomial rings $\Ff_q[u]$ in one variable over a finite field $\Ff_q$.  Take indeed $P_1(u,y)= y^q-y+u$, $P_2(u,y)= (y^q-y)^2+u$. 
For every $m(u)\in \Ff_q[u]$, the constant term of $m(u)^q - m(u)$ is zero, so $P_1(u,m(u))$ and $P_2(u,m(u))$ are divisible by $u$. Also note that a necessary condition for (AV2) is that no prime of $Z$ divides $P_1,\ldots,P_s$ (in $Z[y]$). The following statement is about the converse.
\vskip 1,5mm

\noindent
{\it Addendum to Theorem \ref{thm:main_theorem}}.  {\it In the setup of {\rm Definition \ref{th:schinzel-coprime}}, assume further that  $Z$ has this {\rm residue property}: its quotients by prime principal ideals are infinite. Then  {\rm (AV2)} is equivalent
to $P_1,\ldots,P_s$ having no common prime divisor in $Z$. }

\vskip 1mm

\noindent
Lemma \ref{lem:prelim-proof-R[u]} shows further that the residue pro\-per\-ty holds  
in case {\rm (b)} of {\rm Theorem \ref{thm:main_theorem}}, and in case {\rm (c)} unless $Z=\Ff_q[u]$.
Of course, the residue pro\-per\-ty fails for $Z=\Zz$.
\end{remark}

We note that Theorem \ref{thm:main_theorem} with $Z=\Zz$ (or  \cite[Lemme]{Sc59}), combined with the Dirichlet theorem, yields this {\it mod $N$ version of the Schinzel Hypothesis} (see Section \ref{ssec:proofs_PID}).

\begin{corollary}	
\label{cor:schinzel-modulo}
Let $P_1(y),\ldots, P_s(y)\in \Zz[y]$ be $s$ polynomials ($s\geq 1$) satisfying {\rm Assumption on Values (AV1)}. For every $N\in \Nn^*$, there exist in\-fi\-ni\-tely many 
$m\in \Zz$ such that $P_1(m),\ldots,P_s(m)$ are congruent modulo $N$ to prime numbers not dividing $N$.
\end{corollary}

Schinzel's goal in \cite{Sc59} was to prove this  {\it mod $N$ version of 
 the {\it Goldbach} problem}: 
{for every $N\in \Nn^*$ and every even integer $2n$, infinitely many prime numbers $p,q$ exist such that $2n\equiv p+q$ modulo $N$} (from Corollary \ref{cor:schinzel-modulo}, just take $(P_1,P_2)=(y,\hskip 1pt {2n}-y)$). 

\vskip 1,5mm

The following applications of Theorem \ref{thm:main_theorem} are newer: they are {\it integral} 
versions of Hilbert's irreducibility theorem.

\begin{theorem} \label{thm-intro}
Let $P_1(t,y), \ldots, P_s(t,y)$ be $s$ irreducible polynomials in $\Zz[t,y]$ ($s\geq 1$), of degree at least $1$ in $y$
and satisfying this assumption: 
\vskip 1mm

\noindent
{\rm (AV3)} {\it no prime number $p$ divides all polynomials $\prod_{i=1}^s P_i(m,y)$ with $m\in \Zz$}.

\vskip 1mm

\noindent
Then infinitely many $m\in \Zz$ exist such that $P_1(m,y),\ldots, P_s(m,y)$ are irreducible in $\Zz[y]$.
\end{theorem}

The classical Hilbert irreducibility theorem concludes that for infinitely many $m\in \Zz$, $P_1(m,y),\ldots, P_s(m,y)$ are irreducible {\it in} $\Qq[y]$, instead of $\Zz[y]$. The irreducibility is here over the ring $\Zz$; equivalently, for each of the ``good'' $m$, $P_i(m,y)$ is irreducible in $\Qq[y]$ {\it and} is {\it primitive}, \hbox{i.e.} its coefficients have no common divisor in $\Zz$, $i=1,\ldots,s$. 

Assumption (AV3) is obviously necessary in Theorem \ref{thm-intro}. It fails for example for $s=1$ and $P_1 = (t^p-t)y+ (t^p-t+p)$ with $p$ a prime number. Excluded polynomials $P_1,\ldots,P_s$ are exactly those for which, for some prime $p$, one of the $P_i$ lies in the ideal $\langle t^p-t,p\rangle \subset \Zz[t,y]$. 

Note further that the (excluded) case $\deg_y(P_1)=\ldots = \deg_y(P_s)= 0$ 
of Theorem \ref{thm-intro}
is the original Schinzel Hypothesis.

\vskip 1mm

Next statement is a variant of Theorem \ref{thm-intro} with rings $Z$ as in Theorem \ref{thm:main_theorem}(c). 

\begin{theorem} \label{thm-intro2}
Let $Z$ be a polynomial ring $R[u_1,\ldots,u_r]$   ($r\geq 1$) over a {\rm UFD} $R$,  with $Z\not=\Ff_q[u_1]$. Let $P(t,y)$ be an irreducible polynomial in $Z[t,y]$ of degree at least $1$ in $y$.
Then infinitely many $m\in Z$ exist such that $P(m,y)$ is irreducible in $Z[y]$.
\end{theorem}

Here assumption (AV3) is  automatically satisfied (by Lemma \ref{lem:prelim-proof-R[u]}), but only one polynomial is involved. Statements with several polynomials 
$P_1,\ldots,P_s$
{\it and} avoiding assumption (AV3) exist (for $Z=\Zz$ or $Z=R[u_1,\ldots,u_r]$), but then ``$m\in Z$'' should be replaced by ``$m(y)\in Z[y]$'': in\-finitely many $m(y)\in Z[y]$ exist such that $P_1(m(y),y), \ldots, P_s(m(y),y)$ are irreducible in $Z[y]$; see \cite[Theorem 1.1]{BDN19} and Section \ref{ssec:SH_more_general_rings} below. 

Theorem \ref{thm-intro2} does not extend to case (b) of Theorem \ref{thm:main_theorem}. Take $Z=k[[u_1,\ldots,u_r]]$ with $k$ an infinite field. The final sentence of Theorem \ref{thm-intro2} may fail for some polynomial $P(t,y)$. For $r=1$, it even does with $Z$ and $Z[y]$ replaced by $Q$ and $Q[y]$  \cite[Lemma 15.5.4]{FrJa},
and for $r\geq 2$, with $Z[y]$ replaced by $Q[y]$ (and ``$m\in Z$'' retained) \cite[Example 15.5.6]{FrJa}.

Theorem \ref{thm-intro} and Theorem \ref{thm-intro2} will be proved with $y$ replaced by a tuple $\underline y = (y_1,\ldots, y_n)$ of variables, thus offering irreducibility conclusions for infinitely many fibers of irreducible affine hypersurfaces over ${\rm Spec}\hskip 1pt Z[t]$.

\smallskip

The paper is organized as follows. Section \ref{sec:further} below develops the ``natural approach'', leading to cases (a) and (b) of Theorem \ref{thm:main_theorem} and its consequences, including Theorem \ref{thm-intro} (in the more precise form given in Theorem \ref{thm:HIT-main}).
In Section \ref{ssec:two_Schinzel}, the original Schinzel Hypothesis is extended to our general context and compared 
with the relative version; the Addendum to Theorem \ref{thm:main_theorem} is also proved.  
Section \ref{ssec:SH_more_general_rings} is devoted 
 to the polynomial ring situation and 
the second strategy, leading to case (c) of Theorem \ref{thm:main_theorem} and to Theorem \ref{thm-intro2}.

\section{The natural approach} \label{sec:further}

In Section \ref{ssec:delta}, we introduce a key parameter to the problem. We can then handle cases (a) and (b) 
of Theorem \ref{thm:main_theorem} in Section \ref{ssec:proofs_PID}. Theorem \ref{thm-intro} is proved in Section \ref{ssec:HIT}.

\subsection{The parameter $\delta$} \label{ssec:delta}
Let $Z$ be an integral domain with fraction field $Q$. Denote the set of invertible elements 
of $Z$ (also called units of $Z$) by $Z^\times$.

Let $P_1,\ldots, P_s\in Z[y]$ be $s$ nonzero polynomials  ($s\geq 2)$, with no common divisor in $Q[y]$; equivalently, they have no common root in an algebraic closure of $Q$.

As $Q[y]$ is a PID, we have $\sum_{i=1}^s   P_i\hskip 1pt Q[y]= Q[y]$. It follows that 
$(\sum_{i=1}^s   P_i\hskip 1pt Z[y]) \cap Z$ is a nonzero ideal of $Z$. Fix a nonzero element $\delta$ 
in this ideal. Note that if $s=2$, one can take $\delta$ to be the resultant $\rho={\rm Res}(P_1,P_2)$
(e.g. \cite[V \S 10]{Langoriginal}). In general, we have:
\vskip 1mm

\noindent
(1) \hskip 25mm $V_1(y)P_1(y)+\cdots+V_s(y)P_s(y)=\delta$
\vskip 1mm

\noindent
for some polynomials $V_1,\ldots,V_s \in Z[y]$. The same holds with $y=m\in Z$ and so, for every $m\in Z$,
every common divisor of $P_1(m),\ldots,P_s(m)$ divides $\delta$. 

If in addition $Z$ is a PID, the set $(\sum_{i=1}^s   P_i\hskip 1pt Z[y]) \cap Z$ is a principal ideal of $Z$. In this case 
we may and will pick $\delta$ to be a generator of this ideal. 

\begin{remark}[A periodicity property]  \label{rem:periodic}  {\it For every $m\in Z$, denote the set of common divisors of $P_1(m),\ldots,P_s(m)$
by ${\mathcal D}_m$. The function
\vskip 1mm

\centerline{$m\mapsto {\mathcal D}_m$ $(m\in Z)$} 
\vskip 1mm

\noindent

\noindent
has this periodicity property: for every $m, \ell \in Z$, ${\mathcal D}_m = {\mathcal D}_{m+\ell \delta}$. In particular, if $Z$ is a UFD and $d_m={\rm gcd}(P_1(m),\ldots,P_s(m))$, we have $d_m = d_{m+\ell \delta}$ ($m, \ell \in Z$).}
\vskip 1mm

 This is observed by  Frenkel and Pelik\'{a}n in \cite{FP} in the special case ($s=2$, $Z=\Zz$) with $\delta$ replaced by  the resultant $\rho={\rm Res}(P_1,P_2)$. We adjust their argument\footnote{Our parameter $\delta$ seems more appropriate in the situation of $s>2$ polynomials, for which resultants are not defined. Furthermore, even for $s=2$, our version is a slight improvement since 
$\delta$ can be any element of the ideal $(\sum_{i=1}^s   P_i\hskip 1pt Z[y]) \cap Z$, in particular $\rho$. 
If $Z$ is a PID, $\delta$ divides $\rho$, and $\delta \not=\rho$
in general.
}.

For every $m,\ell \in Z$, we have $P_i(m+\ell \delta) \equiv P_i(m) \pmod {\delta}$, $i=1,\ldots,s$. It follows that the common divisors of $P_1(m),\ldots,P_s(m),\delta$ are the same as those of $P_1(m+\ell \delta)$,$\ldots$, $P_s(m+\ell \delta), \delta$. As both the common divisors of $P_1(m),\ldots,P_s(m)$ and those of $P_1(m+\ell \delta)$, $\ldots$, $P_s(m+\ell \delta)$ divide $\delta$, the conclusion ${\mathcal D}_m = {\mathcal D}_{m+\ell \delta}$ follows.
\end{remark}

\begin{remark}[{\rm on the set of ``good'' $m$}] \label{ex:example2} \label{rem:small_density}
It follows from the periodicity property 
 that the set, say ${\mathcal S}$, of $m\in Z$ such that $P_1(m)$,\ldots,$P_s(m)$ have no common divisor in $Z$ is infinite if it is nonempty, and if $Z$ is infinite. 
The set ${\mathcal S}$ can nevertheless be of arbitrarily small density. Take $Z=\Zz$, $P_1(y) = y$, $P_2(y)=y+\Pi_h$, with $\Pi_h$ ($h\in \Nn$) the product of primes in $[1,h]$. The set ${\mathcal S}$ consists of the integers which are prime to $\Pi_h$. Its density is: 

\vskip 0,3mm
\centerline{$\displaystyle \frac{\varphi(\Pi_h)}{\Pi_h} = \left(1-\frac{1}{2}\right) \cdots \left(1-\frac{1}{p_h}\right)$}
\vskip 0,3mm

\noindent
where $p_h$ is the $h$-th prime number and $\varphi$ is the Euler function. The sequence $\varphi(\Pi_h)/\Pi_h$
tends to $0$ when $h\rightarrow \infty$ (since the series $\sum_{h=0}^\infty 1/p_h$ diverges).
\end{remark}

\subsection{The ``easy'' part of Theorem \ref{thm:main_theorem}} \label{ssec:proofs_PID}
The proof of cases (a) and (b) of  Theorem \ref{thm:main_theorem} rests on a common idea, 
which we develop  below.

\begin{proof}[Proof of cases {\rm (a)} and {\rm (b)} of Theorem \ref{thm:main_theorem}] 
Let $Z$ be a UFD of fraction field $Q$ and let $P_1,\ldots, P_s\in Z[y]$ be some nonzero polynomials, 
with no common divisor in $Q[y]$ and satisfying Assumption on Values {\rm (AV2)}. The main point is to construct 
an element $m\in Z$ such that no prime divisor of $\delta$ divides 
${\rm gcd}(P_1(m),\ldots,P_s(m))$. The conclusion then readily follows: as we know that
 this gcd divides $\delta$, it is $1$, that is: $P_1(m),\ldots,P_s(m)$ have no common divisor in $Z$.
The last two paragraphs of the proof explain how to construct such an element $m\in Z$ in both cases (a) and (b) 
of Theorem \ref{thm:main_theorem}.
\vskip 1mm

\noindent
(a) Assume that $Z$ is a PID. From Assumption (AV2), for every prime divisor $p$ of $\delta$, there 
exists $m_p\in Z$ and $i_p\in \{1,\ldots,s\}$ such that $p$ 
does not divide $P_{i_p}(m_p)$. The Chinese Remainder Theorem provides 
an element $m\in Z$ such that $m\equiv m_p \pmod{p}$ for every prime divisor $p$ of $\delta$. 
For these primes $p$, we have $P_{i_p}(m)\equiv P_{i_p}(m_p) \pmod{p}$.
Hence no prime divisor $p$ of $\delta$ divides ${\rm gcd}(P_1(m),\ldots,P_s(m))$. 
\vskip 1mm

\noindent
(b) Assume that $Z$ is a UFD containing an infinite field $k$. From Assumption (AV2), $P_1,\ldots, P_s$ 
have no common prime divisor in $Z$. 
Thus, for every prime divisor $p$ of $\delta$ in $Z$, there is an index $i_p\in \{1,\ldots,s\}$ such that $\overline P_{i_p} \not= 0$ in $(Z/pZ)[y]$; here we denote by $\overline P_{i_p}$ the polynomial obtained from $P_{i_p}$ by reducing the coefficients modulo $p$. As $Z$ is a UFD, the set of prime divisors $p$ of $\delta$ in $Z$ is finite (modulo $Z^\times$). The list of corresponding polynomials $\overline P_{i_p}$ is finite as well.
Since the field $k\subset Z$ is infinite, one can find infinitely many elements $m\in k$ such that $\overline P_{i_p}(\overline m)\not= 0$ in $Z/pZ$ for every prime divisor $p$ of $\delta$; here we use the injectivity of the map $k \rightarrow Z/pZ$; distinct elements in $k$ remain distinct in $Z/pZ$. Thus, for these $m$, no prime divisor $p$ of $\delta$ divides 
$P_{i_p}(m)$, and so no prime divisor $p$ of $\delta$ divides ${\rm gcd}(P_1(m),\ldots,P_s(m))$. 
\end{proof}

\begin{remark}[A weak form of the relative Schinzel hypothesis]\label{rem:direct_proof} When $Z$ is an arbitrary 
 UFD,
  the Artin-Whaples approximation theorem  \cite[XII \S 1]{Langoriginal} can always be used in the proof above, 
and provides an element $m\in Q$ such that $v_p(m-m_p)>0$ for each prime divisor $p$ of $\delta$, where $v_p$ 
 is the standard discrete 
valuation on $Q$ associated to the prime $p$.
We then obtain the following, under Assumption (AV2):
\vskip 1mm

\noindent
(2) {\it Let $S_\delta \subset Z$ be the multiplicative subset of elements of $Z$ that 
are prime to $\delta$. Then there exists an element $m$ in the fraction ring $S_{\delta}^{-1} Z$ such that 
$P_1(m),\ldots,P_s(m)$ have no common divisor in $S_{\delta}^{-1} Z$.}
\end{remark}

\begin{remark}[A stronger property in the PID case] \label{rem:ref_BDN20}
If $Z$ is a PID, the following property, stronger than the relative Schinzel Hypothesis, can be shown. Consider the set ${\mathcal D}^\ast$ of all $d_m= {\rm gcd}(P_1(m),\ldots,P_s(m))$ with $m\in Z$. The set ${\mathcal D}^\ast$, which is finite modulo $Z^\times$ (since all $d_m$ divide $\delta$), is {\it stable under gcd} \hbox{(i.e.} if $d,d' \in \mathcal{D}^\ast$ then $\gcd(d,d')\in \mathcal{D}^\ast$). Thus the gcd $d^\ast$ of all $d_m$ ($m\in Z$) is in ${\mathcal D}^\ast$. But $d^\ast$ is also the gcd of all values $P_1(m),\ldots,P_s(m)$ with $m\in Z$. So under Assumption (AV2), we have $d^\ast=1$, and  \hbox{``$d^\ast \in {\mathcal D}^\ast$''}  means that there exists $m\in Z$ such that $P_1(m),\ldots,P_s(m)$ have no common divisor in $Z$.
The stability property is detailed in \cite{BDN19b}, where other  results on  ${\mathcal D}^\ast$, e.g. an upper bound for the element $d^\ast$, are established.
\end{remark}

Finally we prove our mod $N$ version of the Schinzel Hypothesis (Corollary \ref{cor:schinzel-modulo}). 

\begin{proof}[Proof of Corollary \ref{cor:schinzel-modulo}] Let $P_1(y),\ldots, P_s(y)\in \Zz[y]$ be $s$ polynomials ($s\geq 1$) and $N>0$ be an integer. 
Assume that no prime number divides $\prod_{i=1}^s P_i(m)$ for every $m\in \Zz$.  Consider the two polynomials ${\mathcal P}_1(y) = \prod_{i=1,\ldots,s}P_i(y)$ and ${\mathcal P}_2(y) = N$. They are nonzero, have no common divisor in $\Qq[y]$, and satisfy Assumption (AV2). It follows from Theorem \ref{thm:main_theorem}(a) (or  \cite[Lemme]{Sc59}) that there exist infinitely many $m\in \Zz$ such that ${\mathcal P}_1(m) = \prod_{i=1,\ldots,s}P_i(m)$ and $N$ have no common divisor. In particular, for such $m$, $P_i(m)$ and $N$ have no common divisor for each $i=1,\ldots,s$, and so by the Dirichlet theorem, there exist prime numbers $p_i$ not dividing $N$ and such that  $P_i(m) \equiv p_i \pmod{N}$. \end{proof}

\subsection{An integral version of Hilbert's irreducibility theorem} \label{ssec:HIT} Theorem \ref{thm:HIT-main} below is a more general version of Theorem \ref{thm-intro}: $y$ is replaced by an $n$-tuple $\underline y = (y_1,\ldots,y_n)$ of variables ($n\geq 1$) and $\Zz$ is replaced by the ring of integers $Z$ of a number field $Q$ of class number $1$. The proof provides an explicit version: the ``good'' integers $m$ can be taken to be all the terms of an arithmetic progression $(am+b)_{m\in \Zz}$ with $a$ precisely described.

Let $Z$ be the ring of integers of a number field $Q$. Given $s$ polynomials $P_1(t,\underline y), \ldots, P_s(t,\underline y)$ irreducible in $Z[t,\underline y]$  ($s\geq 1$) and of degree $\geq 1$ in $\underline y$,
consider the set

\vskip 1,5mm

\centerline{$\displaystyle H_Z(P_1,\ldots,P_s)=\left\{m\in Z \hskip 1mm \left | 
\begin{matrix}
\hskip 1mm P_i(m,\underline y)\hskip 1mm  \hbox{\rm irreducible in } Z[\underline y] \cr
\hskip 1mm \hbox{\rm for each }i=1,\ldots,s,  \hfill \cr
\end{matrix} \right.
\right\}$.}

\vskip 1,5mm
\noindent
Call $H_Z(P_1,\ldots,P_s)$ an {\it integral Hilbert subset of} $Z$. In the classical definition,  the ring $Z$ is replaced by the field $Q$: the {\it Hilbert subset} $H_Q(P_1,\ldots,P_s)$ is the set of all $m\in Q$ such that $P_1(m,\underline y), \ldots, P_s(m,\underline y)$ are irreducible in $Q[\underline y]$. For $m$ to be in $H_Z(P_1,\ldots,P_s)$, $m$ has to be in $Z$ and $P_1(m,\underline y), \ldots, P_s(m,\underline y)$ should be irreducible in $Z[\underline y]$. If $Q$ is of class number $1$ and so $Z$ is a UFD, the latter is equivalent to $P_1(m,\underline y), \ldots, P_s(m,\underline y)$ being irreducible in $Q[\underline y]$ {\it and primitive \hbox{w.r.t. Z}}, \hbox{i.e.} its coefficients have no common divisor in $Z$.

\begin{theorem} \label{thm:HIT-main}
Assume that the number field $Q$ is of class number $1$ (e.g. $Q=\Qq$). Let $H_Z(P_1,\ldots,P_s)$ be an integral subset of $Z$. Assume in addition that, for $P= P_1\cdots P_s$, 
\vskip 1mm

\noindent
{\rm (AV3)} {\it no prime $p$ of $Z$ divides all polynomials $P(m,\underline y)$ with $m\in Z$}.

\vskip 1mm

\noindent
Then there exist $a,b \in Z$, $a\not=0$, such that $H_Z(P_1,\ldots,P_s)$ contains the arithmetic progression $(am+b)_{m\in Z}$.
\end{theorem}

\begin{proof} Let $P_1,\ldots,P_s \in Z[t,\underline y]$ be as in the statement. One may assume that none of them is of the form $P(t,y)= c y$ with $c$ invertible in $Z$; such a polynomial satisfies ``$P(m,y) = cy$ irreducible in $Z[y]$ for every $m\in Z$'' and hence can be removed with no loss from $P_1,\ldots,P_s$.

Consider the set
\vskip 1,5mm

\centerline{$\displaystyle P_Z(P_1,\ldots,P_s)=\left\{m\in Z \hskip 1mm \left | 
\begin{matrix}
\hskip 1mm P_i(m,\underline y)\hskip 1mm  \hbox{\rm primitive \hbox{w.r.t. $Z$}} \cr
\hskip 1mm \hbox{\rm for each }i=1,\ldots,s  \hfill \cr
\end{matrix} \right.
\right\}$}

\vskip 1,5mm
\noindent
Clearly we have $H_Z(P_1,\ldots,P_s) = H_Q (P_1,\ldots,P_s) \cap P_Z(P_1,\ldots,P_s)$.

Fix $i=1,\ldots,s$. As the polynomial $P_i(t,\underline y) $ is irreducible in $Z[t,\underline y]$ and of degree $\geq 1$ in $\underline y$, its nonzero coefficients $P_{i}(t) \in Z[t]$ have no common factor in $Q[t]$. Due to the initial reduction, there are at least two such nonzero coefficients $P_{ij}(t)$. Denote by $\delta_i$ the parameter defined in Section \ref{ssec:delta} associated with these polynomials $P_{ij}(t)$. 

The first stage, which concerns $P_Z(P_1,\ldots,P_s)$, re-uses the approach to Theorem \ref{thm:main_theorem}(a) (see Section \ref{ssec:proofs_PID}). 
From Assumption (AV3), for every prime divisor $p$ of $\delta_1 \cdots \delta_s$, there 
exists $m_p\in Z$ such that $p$ does not divide $P(m_p,\underline y)$. 
Consequently, the prime $p$ does not divide any of the polynomials $P_i(m_p,\underline y)$, i.e. $p$ does not divide some coefficient $P_{ij_i}(m_p)$, $i=1,\ldots,s$. Let $a_0$ be the product of primes dividing $\delta_1 \cdots \delta_s$. The Chinese Remainder Theorem provides an element $b_0\in Z$ such that $b_0 \equiv m_p \pmod{p}$ for every prime divisor $p$ of $a_0$. This property follows: for every $t^\ast \in Z$ such that $t^\ast \equiv b_0 \pmod{a_0}$, no prime divisor of $a_0$ divides any of the polynomials $P_1(t^\ast,\underline y),\ldots,P_s(t^\ast,\underline y)$. But since for each $i=1,\ldots,s$, every prime divisor $p\in Z$ of $P_i(t^\ast,\underline y)\in Z[\underline y]$ must divide $\delta_i$, we obtain that each of the polynomials $P_1(t^\ast,\underline y),\ldots,P_s(t^\ast,\underline y)$ is primitive \hbox{w.r.t.} $Z$. We have thus shown that the set $P_Z(P_1,\ldots,P_s)$ contains the arithmetic progression $(a_0m+b_0)_{m\in Z}$.
\vskip 1mm

For the second stage, we first consider the special case that $\underline y$ is a single variable $y$.

\noindent
{\bf{\textit{1st case:}}} $n=1$.
The second stage combines the conclusion on $P_Z(P_1,\ldots,P_s)$ with Theorem 3.1 from \cite{acta2016}. This result gives some integers $N$, $B$, $C$ and a finite extension $L/\Qq$ with this property:
\vskip 1mm

\noindent
(3) {\it If $p_1, . . . , p_N$ are $N$ distinct prime numbers satisfying: $p_i$ does not divide $B$, $p_i \geq C$ and $p_i$ is totally split in $L/\Qq$ ($i = 1,...,N$), then for $a_1=p_1 \cdots p_N$, there exists $b_1\in \Zz$ such that the Hilbert subset $H_Q (P_1,\ldots,P_s)$ contains the arithmetic progression $(a_1m+b_1)_{m\in \Zz}$.}
\vskip 1mm

\noindent
(In \cite{acta2016}, Theorem 3.1 is stated for a Hilbert subset $H_Q(P_1)$ corresponding to a single polynomial $P_1$, also assumed to be irreducible in $\overline \Qq[t,y]$, but Remark 3.4 there extends the statement to a general Hilbert subset. The constants $N,B,C$ and the extension $L/\Qq$ are explicitly described in \cite[Addendum 3.2]{acta2016}).
\noindent

Pick $p_1, \ldots, p_N$ as indicated in (3) and not dividing $\delta_1 \cdots \delta_s$ (condition regarding $L/\Qq$ can be guaranteed thanks to the Chebotarev Theorem). Let then $a_1, b_1$ be the two integers given by (3). Since $a_0$ and $a_1$ have no common factor, there exists an arithmetic progression $(am+b)_{m\in \Zz}$ with $a=a_0a_1$ contained in both the arithmetic progressions $(a_0m+b_0)_{m\in \Zz}$ and $(a_1m+b_1)_{m\in \Zz}$ (again by the Chinese Remainder Theorem). By construction, the arithmetic progression $(am+b)_{m\in \Zz}$ is hence contained in both $H_Q (P_1,\ldots,P_s)$ and $P_Z(P_1,\ldots,P_s)$, and so in $H_Z(P_1,\ldots,P_s)$.
\vskip 1mm

\noindent
{\bf{\textit{General case: $n\geq 1$.}}} We use a classical reduction from the Hilbert subset theory. 
From \cite[Lemma 12.1.3]{FrJa}, there is a finite set ${\mathcal S}$ of irreducible polynomials in $Q[t,y]$, of degree at least $1$ in $y$ and such that 
\vskip 1mm

\centerline{${H}_{Q}(P_1,\ldots,P_s) \supset H_{Q}({\mathcal S})$, up to some finite subset of $Q$.}
\vskip 1mm

\noindent
Hence, property (3), which holds for $H_Q({\mathcal S})$, holds for the Hilbert subset 
$H_Q(P_1,\ldots,P_s)$ as well. The proof can then be concluded as in first case.
\end{proof}

\section{The two Schinzel Hypotheses} \label{ssec:two_Schinzel} 

We generalize the original Schinzel Hypothesis to
more general rings (Definition \ref{def:schinzel-general}). Proposition \ref{prop:schinzel-implies-coprime} then shows that this generalization is stronger than the relative version, under suitable assumptions. The intermediate Lemma \ref{lem:prelim-proof-R[u]} shows that the Assumptions on Values involved in the two Schinzel Hypotheses are automatically satisfied under some standard condition. This proves in particular the Addendum to Theorem \ref{thm:main_theorem}.

\begin{definition}
\label{def:schinzel-general} 
Let $Z$ be an integral domain with fraction field $Q$. We say that the {\it Schinzel Hypothesis holds for $Z$} if the following is true. Let $P_1,\ldots,P_s \in Z[y]$ be polynomials ($s\geq 1$), irreducible in $Q[y]$, distinct modulo $Q^\times$ and satisfying this  {\it Assumption on Values}: 
\vskip 1mm

\noindent
(AV1) {\it no prime $p$ of $Z$ divides all elements $\prod_{i=1}^s P_i(m)$, $m\in Z$}.
\vskip 1mm

\noindent
Then given any finite set $S$ of primes of $Z$, there is $m\in Z$ such that $P_1(m),\ldots, P_s(m)$ are primes of $Z$, pairwise distinct modulo $Z^\times$, and distinct from every prime in $S$ modulo $Z^\times$.
\end{definition}

Definition \ref{def:schinzel-general} indeed generalizes the original Schinzel Hypothesis (H) (as recalled in Section \ref{sec:intro}). The condition involving $Z^\times$ in the conclusion here takes into account the fact that $Z^\times$ may be infinite in general.

\begin{remark} Clearly the ring $Z$ should have infinitely many distinct primes modulo $Z^\times$ for the Schinzel Hypothesis to hold; hence it does not hold for fields, local rings, etc. Even when this necessary condition is satisfied, the Schinzel Hypothesis does not hold in general: it is known to fail for $Z=\Ff_2[u]$. Take indeed $P_1(u,y) = y^8 +u^3\in \Ff_2[u,y]$; $P_1$ is irreducible in $\Ff_2[u,y]$, satisfies Assumption (AV1) (since for example $P_1(u,0) = u^3$ and $P_1(u,1)=u^3+1$ have no common divisor) and yet, from an example of Swan \cite[pp.1102-1103]{swan}, $m(u)^8 + u^3$ is reducible (hence not prime) in $\Ff_2[u]$ for every $m(u)\in \Ff_2[u]$. 
\end{remark}

\begin{lemma} \label{lem:prelim-proof-R[u]} Let $Z$ be an integral domain and $P_1,\ldots,P_s\in Z[y]$ some nonzero polynomials.

\noindent
{\rm (a)} Assume that the quotients of $Z$ by prime principal ideals are infinite. 
Then: 
\vskip 0,4mm

\noindent
{\rm (a-1)} {\rm Assumption on Values} {\rm (AV1)} holds (with $s\geq 1$) if and only if each of the polynomials $P_1,\ldots,P_s$ has no prime divisor in $Z$.
\vskip 0,4mm

\noindent
{\rm (a-2)} {\rm Assumption on Values} {\rm (AV2)} (with $s\geq 2$) holds if and only if the polynomials $P_1,\ldots,P_s$ have
no common prime divisor in $Z$.

\vskip 1mm

\noindent
{\rm (b)} If $Z=R[u_1,\ldots,u_r]$ is a polynomial ring over an integral domain $R$, and if either 
$R$ is infinite or if $r\geq 2$, then the quotients of $Z$ by prime principal ideals are infinite. The
same conclusion holds if $Z$ is an integral domain containing an infinite field.
\end{lemma}

On the other hand, $\Zz$ is a typical example of ring that has {\it finite} quotients by prime principal ideals. Conclusions
in (a) are in fact false for $Z=\Zz$ (Remark \ref{rem:AV2}).

\begin{proof}
(a-1) 
Assume that (AV1) fails, \hbox{i.e.} there is a prime $p\in Z$ such that $\prod_{i=1}^s P_i(m) = 0$ in $Z/pZ$ for every $m\in Z$ and $i=1,\ldots,s$. The quotient ring $Z/pZ$ is an integral domain and, under the assumption on $Z$, is infinite. Therefore the polynomial $\prod_{i=1}^s P_i$ is zero in $(Z/pZ)[y]$. As this ring is an integral domain, there is an index $i$ in $\{1,\ldots,s\}$ such that $P_i$ is zero in $(Z/pZ)[y]$. Conclude that $p$ is a prime divisor of $P_i$. 
Conversely, if $p\in Z$ is a prime divisor of $P_i$ for some $i\in \{1,\ldots,s\}$, it is a common prime 
divisor of all $\prod_{i=1}^s P_i(m)$ with $m\in Z$, and so Assumption (AV1) fails.

\vskip 1mm

\noindent
(a-2) 
Assume that (AV2) fails, \hbox{i.e.} there is a prime $p\in Z$ such that $P_i(m) = 0$ in $Z/pZ$ for every $m\in Z$ and $i=1,\ldots,s$. The quotient ring $Z/pZ$ is integral and infinite.  Hence $P_i = 0$ in $(Z/pZ)[y]$, $i=1,\ldots,s$. Conclude that $p$ is a common prime divisor of $P_1,\ldots,P_s$. Conversely, if $p\in Z$ is a common prime divisor of $P_1,\ldots,P_s$, it is a common prime divisor of all values $P_1(m),\ldots,P_s(m)$ with $m\in Z$, and so Assumption (AV2) fails.

\vskip 1mm

\noindent
(b) With $Z=R[u_1,\ldots,u_r]$ as in the statement, assume further that $R$ is infinite. Let $p\in R[u_1,\ldots,u_r]$ be a prime element. Suppose first that $p\not\in R$, say $d=\deg_{u_r}(p)\geq 1$. 
The elements $1,u_r,\ldots,u_r^{d-1}$ are $R[u_1,\ldots,u_{r-1}]$-linearly independent in the integral domain $Z/pZ$.
As $R$ is infinite, the elements $\sum_{i=0}^{d-1} p_i u_r^i$ with $p_0,\ldots,p_{d-1}\in R$ are  infinitely many different elements in $Z/p Z$. Thus $Z/pZ$ is infinite. 
In the case that $p\in R$, the quotient ring $Z/pZ$ is $(R/pR)[u_1,\ldots,u_r]$, which is infinite too. 

If $Z=R[u_1,\ldots,u_r]$ with $r\geq 2$, write $R[u_1,\ldots,u_r] = R[u_1][u_2,\ldots,u_r]$ and use the previous
paragraph with $R$ taken to be the infinite ring $R[u_1]$.

If $Z$ is an integral domain containing an infinite field $k$, the containment $k\subset Z$ in\-du\-ces an injective morphism $k\hookrightarrow Z/{\frak p}$ for every prime ideal ${\frak p}\subset Z$. The last claim follows.
\end{proof}

\begin{proposition} \label{prop:schinzel-implies-coprime}
Let $Z$ be a UFD with the property that  the quotients of $Z$ by prime principal ideals are infinite.
If the Schinzel Hypothesis holds for $Z$, then the relative Schinzel Hypothesis holds for $Z$.
\end{proposition}

\begin{proof} 
Let $P_1,\ldots,P_s\in Z[y]$  be $s\geq 2$ nonzero polynomials, with no common divisor  in $Q[y]$, and satisfying Assumption on Values (AV2); equivalently, $P_1,\ldots,P_s$ have no common prime divisor in $Z$ (Lemma \ref{lem:prelim-proof-R[u]}).

Consider, for $i=1,\ldots,s$, the prime factorization
\vskip 1mm

\centerline{$P_i(y) = \prod_{j\in I_i} {P}_{ij}(y)^{\beta_{ij}}$}
\vskip 1mm

\noindent
of $P_i$ in the UFD $Z[y]$; the set $I_i$ is finite, the ${P}_{ij}$ are irreducible in $Z[y]$, pairwise distinct modulo $Z^\times$, and the $\beta_{ij}$ are positive integers. Denote  the finite list of polynomials $P_{ij}(y)$ ($j\in I_i$, $i=1,\ldots,s$) that are of positive degree in $y$ by ${\mathcal L}$. 

The polynomials ${\sf P}\in {\mathcal L}$, being irreducible in $Z[y]$ and of degree $\geq 1$, are irreducible in $Q[y]$ and have no prime divisor in $Z$ (Gauss's lemma). From Lemma \ref{lem:prelim-proof-R[u]}, they satisfy Assumption on Values {\rm (AV1)}. As we assume that the Schinzel Hypothesis holds for $Z$, we can conclude that there exist infinitely many $m \in Z$ such that the values ${\sf P}(m)$ with ${\sf P}\in {\mathcal L}$ are primes of $Z$, are pairwise distinct modulo $Z^\times$, and are different modulo $Z^\times$ from
the constant polynomials $P_{ij}$ involved in the prime factorizations of $P_1,\ldots,P_s$ above. 
Conclude that for those $m\in Z$,
\vskip 1mm

\centerline{$P_i(m) = \prod_{j\in I_i} P_{ij}(m)^{\beta_{ij}}$}
\vskip 1mm

\noindent
is a prime factorization of $P_i(m)$ in  $Z$, $i=1,\ldots,s$, with the property that 
\vskip 1mm

\centerline{$P_{ij}(m) \not= P_{i^\prime j^\prime}(m)$, modulo $Z^\times$, if $(i,j)\not=(i^\prime,j^\prime)$.}
\vskip 1mm

\noindent
For such choices of $m$, $P_1(m), \ldots, P_s(m)$ have 
no common divisor in $Z$.
\end{proof}

\section{The polynomial ring case} \label{ssec:SH_more_general_rings}

If $R$ is not a field, polynomial rings $R[\underline u]$ do not fall into case (a) or (b) of Theorem \ref{thm:main_theorem}. A recent new tool is however available for them: Theorem 1.1 of \cite{BDN19} shows that the Schinzel Hypothesis itself, in the form given in Definition \ref{def:schinzel-general}, holds for polynomial rings $Z=R[u_1,\ldots,u_r]$ over a UFD $R$ if this condition holds:
\vskip 0,5mm

\noindent
(1) {\it $r\geq 1$, $K={\rm Frac}(R)$ has {a product formula}, and $K$ is imperfect if of \hbox{characteristic $p>0$.}}
\vskip 1mm

For {\it fields with the product formula}, we refer to \cite[\S 15.3]{FrJa}. 
The basic example is $\Qq$. The product formula is: $\prod_{p} |a|_p \cdot |a| = 1$
for every $a\in \Qq^*$, where $p$ ranges over all prime numbers, $|\cdot|_p$ is the $p$-adic absolute value 
and $|\cdot |$ is the standard absolute value. Rational function fields $k(u_1,\ldots,u_r)$ in $r\geq 1$ variables 
over a field $k$ and finite extensions of fields with the product formula are other examples. Also recall that a  field $K$ of characteristic $p>0$ is {\it imperfect} if $K^p\not = K$.

Conjoining \cite{BDN19} and Proposition \ref{prop:schinzel-implies-coprime}, we obtain that the relative Schinzel Hypothesis holds for polynomial rings $R[u_1,\ldots,u_r]$ satisfying (1). This is in particular always true if $r\geq 2$: write $R[u_1,\ldots,u_r]=R[u_1][u_2,\ldots,u_r]$ and note that (1) is satisfied with $R$ taken to be $R[u_1]$. This however leaves out some polynomial rings of interest, for example, polynomial rings $R[u]$ in one variable over a discrete valuation ring $R$, e.g. $\Zz_p[u]$, $\Cc[[x]][u]$. 
 
Case (c) of Theorem \ref{thm:main_theorem} includes these missing cases, in fact includes all polynomials rings over a UFD. The proof given below also rests on \cite{BDN19}, but uses some more precise results than \cite[Theorem 1.1]{BDN19}.

\subsection{The main ingredient} In order to prove Theorem \ref{thm:main_theorem}(c), we could restrict to the case of polynomial rings in one variable ($r=1$) for two reasons: first, because an obvious inductive argument gives the general case $r\geq 1$; second, because the case $r\geq 2$ is already known, as just explained above.
However our proof is somehow constructive and we prefer to keep this feature in the general case $r\geq 1$. 

Theorem \ref{thm:main_theorem}(c)  will be deduced from this specific statement.

\begin{proposition} \label{prop:main_general}
Let $R_0$, $R$ be two integral domains with $R_0\subset R$ and $K_0$, $K$ be their fraction fields. Assume that $K_0$ is a field with the product formula, and is imperfect if of characteristic $p>0$. Let $\underline u=(u_1,\ldots,u_r)$ be an $r$-tuple of variables ($r\geq 1$).
Let $P_1,\ldots,P_s\in R_0[\underline u,y]$ be $s\geq 2$ nonzero polynomials, with no common divisor in $K[\underline u,y]$ and no common prime divisor $p\in R$.
Then there exists $m\in R_0[\underline u]$ such that 
$P_1(\underline u,m(\underline u))$,$\ldots$, $P_s(\underline u,m(\underline u))$ have no common divisor in $K[\underline u]$ and no common prime divisor $p\in R$.
\end{proposition}

\begin{remark} \label{rem:final_argument} This remark will be used several times in the rest of the paper.
\vskip 0,5mm

\noindent
(a)  In general, {\it given an integral domain $A$ of fraction field $C$, if polynomials $Q_1,\ldots,Q_s \in A[\underline u]$ have no common divisor in $C[\underline u]$ and no common prime divisor $p\in A$, then they have no common prime divisor in $A[\underline u]$}. Indeed, a common prime divisor in $A[\underline u]$, say $\pi$, of $Q_1,\ldots,Q_s$ cannot be of degree $\geq 1$ as then $\pi$ would have a prime divisor  of degree $\geq 1$ in the UFD $C[\underline u]$, which itself would be a common prime divisor in $C[\underline u]$ of $Q_1,\ldots,Q_s$; and $\pi$ cannot be of degree $0$, as then it would be a common prime divisor in $A$ of $Q_1,\ldots,Q_s$. 

\vskip 1mm

\noindent
(b) Part (a) above shows that the conclusion of Proposition \ref{prop:main_general}, 
\hbox{viz.} ``{\it $P_1(\underline u,m(\underline u))$,$\ldots$, $P_s(\underline u,m(\underline u))$ have no common divisor in $K[\underline u]$ and no common prime divisor $p\in R$}'' implies that ``{\it $P_1(\underline u,m(\underline u)), \ldots, P_s(\underline u,m(\underline u))$ have {\it no common {\rm prime} divisor} in $R[\underline u]$}''. 
If $R$ is a UFD (which is not assumed in Proposition \ref{prop:main_general}), the two conditions are equivalent, 
and equivalent to {\it $P_1(\underline u,m(\underline u)), \ldots, P_s(\underline u,m(\underline u))$ having {\it no common divisor} in $R[\underline u]$} (prime or not), 
\hbox{i.e.} the conclusion of the relative Schinzel Hypothesis for $Z=R[\underline u]$.
\end{remark}

We assume more generally in the proof below that $R_0$ is a {\it Hilbertian ring}. 

\begin{definition} \label{def:HilbertUFD} Let $R_0$ be an integral domain such that the fraction field $K_0$ is imperfect if of characteristic $p>0$. The ring $R_0$ is said to be a {\it Hilbertian ring} if every Hilbert subset of $K_0$ \hbox{\rm (Definition is recalled in Section \ref{ssec:HIT})}  contains elements of $R_0$. \end{definition}

The original definition from \cite[\S 13.4]{FrJa} has the defining condition only requested for ``separable Hilbert subsets'' but 
\cite[Proposition 4.2]{BDN19} shows that it is equivalent to Definition \ref{def:HilbertUFD} under the imperfectness condition.
\cite[Theorem 4.6]{BDN19} shows further that assumptions on $R_0$ from Proposition \ref{prop:main_general}, i.e. $K_0={\rm Frac}(R_0) $ being a field with the product formula, imperfect if of characteristic $p>0$, imply that $R_0$ is a {Hilbertian ring}.

\begin{proof}[Proof of Proposition \ref{prop:main_general}]
The polynomials $P_1,\ldots,P_s$ are in $K_0[\underline u,y]$, and since they have no common divisor in $K[\underline u,y]$, they also have no common divisor in $K_0[\underline u,y]$.
Consider, for $i=1,\ldots,s$, the prime factorization
\vskip 1mm

\centerline{$P_i(\underline u,y) = \prod_{j\in I_i} {P}_{ij}(\underline u,y)^{\beta_{ij}}$}
\vskip 1mm

\noindent
of $P_i(\underline u,y)$ in the UFD $K_0[\underline u,y]$; the set $I_i$ is finite, the ${P}_{ij}$ are irreducible in $K_0[\underline u,y]$, pairwise distinct modulo $K_0^\times$, and the $\beta_{ij}$ are positive integers. Denote  the finite list of all polynomials $P_{ij}(\underline u,y)$ ($j\in I_i$, $i=1,\ldots,s$) that are of positive degree in $y$ by ${\mathcal L}$. 

Let $\ell\geq 2$ be an integer, $\underline \lambda=(\lambda_0, \lambda_1,\ldots,\lambda_\ell)$ be a $(\ell+1)$-tuple of variables and $Q_0, Q_1,\ldots,Q_\ell$ be $(\ell+1)$ distinct monic monomials in $R_0[\underline u]$ such that $Q_0=1$ and both $Q_1$ and $Q_2$ are of degree $\deg(Q_i) > \max_{1\leq i\leq s} \deg_{\underline u}(P_i)$.

For each ${\sf P}\in {\mathcal L}$, write:
\vskip 1,5mm

\centerline{$\left\{\begin{matrix}
{\sf P}(\underline u,y)= {\sf P}_{\rho}(\underline u) \hskip 2pt y^{\rho} +\cdots+ {\sf P}_{1}(\underline u) \hskip 1pt y + {\sf P}_{0}(\underline u)$ \hskip 5mm ($\rho = \deg_y({\sf P}) \geq 1) \hfill \cr
M(\underline \lambda,\underline u) = \sum_{j=0}^{\ell} \lambda_j \hskip 1pt Q_j(\underline u) \hfill \cr
{\sf F}(\underline \lambda,\underline u) 
= {\sf P}(\underline u, M(\underline \lambda,\underline u))= {\sf P}(\underline u, \sum_{j=0}^\ell \lambda_j Q_j(\underline u)) \hfill \cr
\end{matrix}\right.$
}
\vskip 1,5mm

\noindent Each polynomial $\sf F$ is irreducible in $K_0[\underline \lambda, \underline u]$ \cite[Lemma 2.1]{BDN19}. 
As we assume that $R_0$ is a Hilbertian ring, and is imperfect if of characteristic $p>0$, we may apply \cite[Proposition 4.2]{BDN19} to produce a $(\ell+1)$-tuple $\underline \lambda^\ast= (\lambda_0^\ast, \ldots, \lambda_{\ell+1}^\ast)$ with nonzero coordinates in $R_0$ such that $\lambda_2^\ast \equiv 1 \pmod{\lambda_1^\ast}$ and
for every ${\sf P}\in {\mathcal  L}$, the corresponding polynomial ${\sf F}(\underline \lambda^\ast,\underline u)$ is 
irreducible in $K_0[\underline u]$.

Set $m(\underline u) = M(\underline \lambda^\ast,\underline u)$. The polynomial $m(\underline u)$ is in $R_0[\underline u]$ and has this first property:
for each ${\sf P}\in {\mathcal L}$, the polynomial ${\sf P}(\underline u, m(\underline u))$ is irreducible in $K_0[\underline u]$.
Obviously, for the irreducible polynomials $P_{ij}$ of degree $0$ in $y$ involved in the prime factorizations of $P_1,\ldots,P_s$ in $K_0[\underline u,y]$, the polynomial $P_{ij}(\underline u,m(\underline u))$ is irreducible in $K_0[\underline u]$ as well. Conclude that
\vskip 1mm

\centerline{$P_i(\underline u,m(\underline u)) = \prod_{j\in I_i} {P}_{ij}(\underline u,m(\underline u))^{\beta_{ij}}$}
\vskip 1mm

\noindent
is a prime factorization of $P_i(\underline u,m(\underline u))$ in the UFD $K_0[\underline u]$, $i=1,\ldots,s$. 

Furthermore, inequality 
\vskip 1mm

\centerline{$\displaystyle \deg(m) \geq \deg(Q_1)>\max_{1\leq i\leq s} \deg_{\underline u}(P_i)$}
\vskip 1mm

\noindent
shows that all polynomials $P_{ij}(\underline u,m(\underline u))$ ($j\in I_i$, $i=1,\ldots,s$) 
are distinct modulo $K_0^\times$.

This follows from this standard argument: if $Q(\underline u,y) = \sum_{i=0}^\nu q_i(\underline u)\hskip 1pt  y^i \in R[\underline u,y]$ is such that $Q(\underline u,m(\underline u)) = 0$ with $\deg_{\underline u} (m) >  \deg_{\underline u}(Q)$, then all terms 
$q_i(\underline u) \hskip 1pt m(\underline u)^i$, $i=0,1,\ldots,\nu$, are of pairwise distinct degrees and so must all be zero, thus implying that $Q(\underline u,y)$ itself must be zero.  Applied with $Q=P_{ij} - cP_{i^\prime j^\prime}$ ($c\in K_0^\times$), this argument proves the claim.

Thus we obtain that the polynomials $P_1(\underline u,m(\underline u)), \ldots, P_s(\underline u,m(\underline u)) $ have no common divisor in $K_0[\underline u]$.
But then they also have no common divisor in $K[\underline u]$. This follows from the B\'ezout theorem when $\underline u$ is a single variable; an argument for the general case $r\geq 1$ is given in the proof of
\cite[Lemma 2.1]{BDN19}.

Finally we should show that the polynomials $P_1(\underline u,m(\underline u)), \ldots, P_s(\underline u,m(\underline u))$, which are in $R[\underline u]$, have no common prime divisor $p\in R$. This follows from Lemma \ref{lem:primitive} below. Note that by construction, the polynomial $m(\underline u)$ satisfies its assumptions  (i) and (ii).  If some $p\in R$ were a common prime divisor of each 
$P_i(\underline u,m(\underline u))$, $i=1,\ldots,s$, Lemma \ref{lem:primitive} would conclude that 
$p$ is a common prime divisor of $P_1,\ldots,P_s$, contrary to our assumption.
\end{proof}

\begin{lemma} \label{lem:primitive}
Let $R$ be an integral domain with fraction field $K$ and $p$ be a prime of $R$.
Let $\underline u=(u_1,\ldots,u_r)$ be an $r$-tuple of variables ($r\geq 1$) and $P\in R[\underline u,y]$ 
be a polynomial not divisible by $p$. Let $m(\underline u) \in R[\underline u]$ be a polynomial such that two 
of its monomials $m_1(\underline u)$, $m_2(\underline u)$, with corresponding coefficients $\mu_1,
\mu_2\in R$, satisfy:
\vskip 1,5mm

\centerline{$\left\{\begin{matrix}
{\hbox{\rm (i)}}\hskip 5mm  \mu_2 \equiv 1 \pmod{\mu_1} \hfill \cr
{\hbox{\rm (ii)}}\hskip 3mm \min(\deg(m_1),\deg(m_2)) > \deg_{\underline u}(P). \hfill \cr
\end{matrix}\right.$
}
\vskip 1,5mm

\noindent
Then the polynomial $P(\underline u,m(\underline u))$ is not divisible by the prime $p$.
\end{lemma}

\begin{proof} Set ${\frak p}=pR$. The proof works more generally if ${\frak p}$ is a prime ideal (not necessarily principal)
not containing all coefficients of $P$.
Write 
\vskip 1,5mm

\centerline{$P = P_{\rho}(\underline u) y^\rho + \cdots + P_{1}(\underline u) y + P_{0}(\underline u)$}
\vskip 1,5mm

\noindent
with $P_0,P_1,\ldots,P_\rho\in R[\underline u]$ and $P_\rho \not= 0$. The result being trivial if $\rho =0$, also assume $\rho\geq 1$.

Assume that all coefficients of $P(\underline u,m(\underline u))$ are in
${\frak p}$. Thus we have:
\vskip 1,5mm

\noindent
(2) \hskip 22mm $\overline P_{\rho}(\underline u) \hskip 2pt \overline m(\underline u)^{\rho} +\cdots+ \overline P_{1}(\underline u) \hskip 1pt \overline m(\underline u) +  \overline P_{0}(\underline u) = 0$,
\vskip 1,5mm

\noindent
where we use the notation $\overline U$ for the class modulo ${\frak p}$ of polynomials $U$ with coefficients in $R$. 
It follows from the assumption $P\not =0$ in $(R/{\frak p})[\underline u,y]$ that there is an index, say $j$, in $\{0,1,\ldots, \rho\}$ such that $\overline P_j (\underline u)\not= 0$ in $(R/{\frak p})[\underline u]$.

From assumption (i), we have that $\overline \mu_1$ or $\overline \mu_2$ is nonzero in $R/\frak p$.
Therefore both polynomials $\overline m(\underline u)$ and $\overline P_{j}(\underline u) \hskip 2pt \overline m(\underline u)^{j}$ are nonzero in $(R/\frak p)[\underline u]$.
Furthermore we have:
\vskip 1,5mm

\noindent
(3) \hskip 20mm $\deg (\overline m) \geq \min (\deg(m_1), \deg(m_2))$.
\vskip 1,5mm

The final argument below shows that all nonzero terms $\overline P_{h}(\underline u) \hskip 2pt \overline m(\underline u)^{h}$ with $h\in \{0,1,\ldots, \rho\}$ are of different degrees. This clearly contradicts identity (2).

Assume that two nonzero polynomials $\overline P_{h}(\underline u) \hskip 2pt \overline m(\underline u)^{h}$ and  $\overline P_{k}(\underline u) \hskip 2pt \overline m(\underline u)^{k}$ are of the same degree, for some $h,k \in \{0,1,\ldots, \rho\}$ with $k>h$. Then we have:
\vskip 1,5mm

\centerline{$\displaystyle \deg_{\underline u}(P) = \max_{0\leq i\leq \rho} \deg({P}_{i}) \geq \deg(\overline{P}_{h})-\deg(\overline{P}_{k}) = (k-h) \deg(\overline m) \geq \deg(\overline m)$.}
\vskip 1mm

\noindent
But this, conjoined with (3), contradicts assumption (ii) of the statement.
\end{proof}

\subsection{The final proof} \label{ssec:final_proof} 

\begin{proof}[Proof of Theorem \ref{thm:main_theorem}(c)]
Let $\underline u= (u_1,\ldots,u_r)$ be an $r$-tuple of variables ($r\geq 1$), $R$ be a UFD of fraction field $K$
and  $Z=R[\underline u]$. 
Let $P_1,\ldots,P_s \in Z[y]$ be nonzero polynomials ($s\geq 2$), with no common divisor in $K(\underline u)[y]$ and
satisfying Assumption on Values (AV2).
\vskip 0,5mm

We distinguish four cases.
\vskip 1mm

\noindent
{\it 1st case}: $R$ is of characteristic $0$. 

\noindent
Consider the subring $R_0\subset R$ generated by all the coefficients 
of  $P_1,\ldots,P_s$ as polynomials in $u_1,\ldots,u_r,y$. The fraction field $K_0$ is an extension of $\Qq$ of finite type, and so a finite extension of a field of the form $\Qq(\underline t)$ where $\underline t$ is a finite set, possibly empty, of elements of $K_0$, algebraically independent over $\Qq$. The field $\Qq(\underline t)$ is a field with the product formula: if ${\underline t}=\emptyset$, it is $\Qq$, and if  ${\underline t}\not=\emptyset$, it is isomorphic to a rational function field. Thus $K_0$, which is a finite extension of $\Qq(\underline t)$, is a field with the product formula as well. 

The polynomials $P_1,\ldots,P_s$ are in $R_0[\underline u,y]$, 
have no common divisor in $K(\underline u)[y]$, and from Assumption (AV2),
 they have no common prime divisor $p\in R[\underline u]$ (Lemma \ref{lem:prelim-proof-R[u]}). Hence they have no common 
prime divisor in $R[\underline u][y]$ (Remark \ref{rem:final_argument}). Since $R$ is assumed to be a UFD, it follows that they have no common divisor in $K[\underline u][y]$ (Gauss's lemma). As they have no common prime divisor $p\in R$, the assumptions of Proposition \ref{prop:main_general} hold. Thus we can conclude that there is a polynomial $m\in R_0[\underline u]$ such that $P_1(\underline u,m(\underline u)), \ldots, P_s(\underline u,m(\underline u))$ have no common divisor in $K[\underline u]$ and no common prime divisor $p\in R$. As $R$ is a UFD, this implies that $P_1(\underline u,m(\underline u)), \ldots, P_s(\underline u,m(\underline u))$ have no common divisor in $R[\underline u]$ (Remark \ref{rem:final_argument}).

\vskip 1mm
\noindent
{\it 2nd case}: $R$ is of characteristic $p>0$ and $R$ is not algebraic over its prime field $\Ff_p$. 

\noindent
Consider the subring $R_0\subset R$ generated by all the coefficients 
of  $P_1,\ldots,P_s$ as polynomials in $u_1,\ldots,u_r,y$ and by some element $\theta \in R$ not algebraic over $\Ff_p$.
The fraction field $K_0$ is an extension of $\Ff_p$ of finite type, and so a finite extension of a field $\Ff_p(\underline t)$ where $\underline t$ is a finite set of elements of $K_0$, algebraically independent over $\Ff_p$. Furthermore, $K_0$ is not algebraic over $\Ff_p$ (as it contains $\theta$), so we have 
$\underline t \not= \emptyset$. Hence the field $\Ff_p(\underline t)$ is a field with a product formula, and so $K_0$, 
which is a finite extension of $\Ff_p(\underline t)$, is a field with the product formula as well. Furthermore, $K_0$ is imperfect:
if $t\in {\underline t}$, then $t\in K_0$ but $t\notin K_0^p$.

Use Proposition \ref{prop:main_general} as in 1st case to conclude that there is a polynomial $m\in R_0[\underline u]$ such that $P_1(\underline u,m(\underline u)), \ldots, P_s(\underline u,m(\underline u))$ have no common divisor in $K[\underline u]$ and have no common prime divisor $p\in R$, and so, again by Remark \ref{rem:final_argument},  
have no common divisor in $Z$.

\vskip 1mm
\noindent
{\it 3rd case}: $R$ is of characteristic $p>0$ and $r\geq 2$. 

\noindent
Consider the subring $R_0\subset R[u_1]$ generated by  $u_1$ and all the coefficients 
of  $P_1,\ldots,P_s$ as polynomials in $u_1,\ldots,u_r,y$. Check as in 2nd case that 
the fraction field $K_0$ is a field with the product formula and is imperfect.

The polynomials $P_1,\ldots,P_s$ are in $R_0[u_2,\ldots,u_r,y]$. Here we apply Proposition \ref{prop:main_general} with the ring $R$ there taken to be the ring $R[u_1]$ here.
Observe as in 1st case that $P_1,\ldots,P_s$ have no common prime divisor in $R[u_1][u_2,\ldots,u_r,y]$.
It follows that they have no common divisor in $K[u_1][u_2,\ldots,u_r]$
and none in $K(u_1)[u_2,\ldots,u_r]$ either (using that $R$ is a UFD and Gauss's lemma). Furthermore they have no common prime divisor $p\in R[u_1]$: 
the case $\deg(p) \geq 1$ is ruled out by $P_1,\ldots,P_s$ having no common divisor in $K[u_1][u_2,\ldots,u_r]$, and
the case $\deg(p) =0 $ by Assumption (AV2).
Conclude from Proposition \ref{prop:main_general} that there is a polynomial $m\in R_0[u_2,\ldots,u_r]$ such that $P_1(\underline u,m(\underline u)), \ldots, P_s(\underline u,m(\underline u))$ 
have no common divisor in $K(u_1)][u_2,\ldots,u_r]$ and no common prime divisor $p\in R[u_1]$.
It follows that they have no common divisor in $K[u_1][u_2,\ldots,u_r]$ and no common prime divisor $p\in R$, and also, 
by Remark \ref{rem:final_argument}, 
that they have no common divisor in $R[u_1][u_2,\ldots,u_r]=Z$.

\vskip 1mm
\noindent
{\it 4th case}: $R$ is of characteristic $p>0$, $R$ is algebraic over its prime field $\Ff_p$ and $r=1$. 
The ring $R$ is then necessarily a field (as it is integral over a field). The ring $Z=R[\underline u] = R[u_1]$ 
is then a PID. The required conclusion follows from the already proven Theorem \ref{thm:main_theorem}(a).
\end{proof}

\begin{remark} \label{rem:final}
The assumption that $R$ is a UFD is used twice: first, to guarantee that the polynomials $P_1,\ldots,P_s$, having no common divisor in $R[\underline u,y]$, have no common divisor in $K[\underline u,y]$; second, to conclude with Remark \ref{rem:final_argument} that $P_1(\underline u,m(\underline u)), \ldots, P_s(\underline u,m(\underline u))$ have no common divisor in $R[\underline u]$. 

Regarding the former, we can replace $R$ being a UFD by the assumption that $P_1,\ldots,P_s$ have no common divisor in $K[\underline u,y]$ (which is stronger than the assumption of the relative Schinzel Hypothesis that they have no common divisor in $K(\underline u)[y]$).  Regarding the latter, we can stop the proof right after applying Proposition \ref{prop:main_general} and before using that $R$ is a UFD, \hbox{i.e.} conclude that $P_1(\underline u,m(\underline u)), \ldots, P_s(\underline u,m(\underline u))$ have no common divisor in $K[\underline u]$ and no common prime divisor $p\in R$. 
Thus we obtain this generalization of Theorem \ref{thm:main_theorem}{\rm (c)} to polynomial rings $R[\underline u]$ with coefficient rings $R$ that are not necessarily UFDs:
\vskip 1,5mm

\noindent
{\it Let $R$ be an integral domain with fraction field $K$ and $P_1,\ldots,P_s \in R[\underline u,y]$ be nonzero polynomials ($s\geq 2$), with no common divisor in $K[\underline u,y]$ and satisfying this {assumption}:
\vskip 1,5mm

\noindent
{\rm (AV2)} {\it no prime of $R[\underline u]$ divides all polynomials $P_1(\underline u,m(\underline u)), \ldots, P_s(\underline u,m(\underline u))$ ($m\in R[\underline u]$)}.

\vskip 1mm

\noindent
Then there exists $m\in R[\underline u]$ such that $P_1(\underline u, m(\underline u))$,\ldots,$P_s(\underline u, m(\underline u))$ have no common divisor in $K[\underline u]$ and no common prime divisor $p\in R$, and hence no common prime divisor in $R[\underline u]$.}
\end{remark}

\subsection{Proof of Theorem \ref{thm-intro2}} We will deduce it from the following generalization.

\begin{proposition} \label{thm-intro2-general}
Let $Z$ be a {\rm UFD} and assume further that $Z$ is a Hilbertian ring,  imperfect if of characteristic $p>0$ and satisfies the relative Schinzel Hypothesis. Let $\underline y = (y_1,\ldots,y_n)$ be an $n$-tuple of variables ($n\geq 1$) and $P(t,\underline y)$ be an irreducible polynomial in $Z[t,\underline y]$ of degree at least $1$ in $\underline y$ and such that
\vskip 1mm

\noindent
{\rm (AV3)} {\it no prime $p$ of $Z$ divides all polynomials $P(m,\underline y)$ with $m\in Z$}.

\vskip 1mm

\noindent
Then infinitely many $m\in Z$ exist such that 
$P(m,\underline y)$ is irreducible in $Z[\underline y]$.
\end{proposition}

\begin{proof}[Proof of Proposition \ref{thm-intro2-general}] Denote the nonzero coefficients in $Z[t]$ of $P(t,\underline y)$, viewed as a polynomial in $\underline y$, by $P_1(t),\ldots,P_s(t)$. As noted in the proof of Theorem \ref{thm:HIT-main}, one may assume with no loss that $s\geq 2$. Set $Q={\rm Frac}(Z)$. The polynomials $P_1(t),\ldots,P_s(t)$ have no common divisor in $Q[t]$ (since $P(t,\underline y)$ is irreducible in $Z[t,\underline y]$ and $Z$ is a UFD) and they satisfy assumption (AV2) (since $P(t,\underline y)$ satisfies (AV3)). It follows from the relative Schinzel Hypothesis that there exists $m_0\in Z$ such that $P_1(m_0),\ldots,P_s(m_0)$ have no common divisor in $Z$; equivalently, $P(m_0,\underline y)$ is primitive \hbox{w.r.t.} $Z$. 
As in Section \ref{ssec:delta}, let $\delta$ be a nonzero element of the ideal $(\sum_{i=1}^s P_i(t) Z[t]) \cap Z$. From Remark \ref{rem:periodic}, for every $\ell \in Z$, $P_1(m_0 + \ell \delta),\ldots,P_s(m_0 + \ell \delta)$ have no common divisor in $Z$, i.e. $P(m_0 + \ell \delta,\underline y)$ is primitive \hbox{w.r.t.} $Z$ as well. The polynomial $P(m_0+\delta t,\underline y)$ is in $Z[t,\underline y]$ and is irreducible in $Q[t,\underline y]$. As $Z$ is a Hilbertian ring, imperfect if of characteristic $p>0$, infinitely many $\ell \in Z$ exist such that $P(m_0 + \ell \delta,\underline y)$ is irreducible in $Q[\underline y]$. As these polynomials are primitive \hbox{w.r.t.} $Z$, they are irreducible in $Z[\underline y]$.
\end{proof}

\begin{proof}[Proof of Theorem \ref{thm-intro2}] In Theorem \ref{thm-intro2}, $Z$ is a polynomial ring $R[u_1,\ldots,u_r]$ ($r\geq 1$) over a UFD $R$. Thus $Z$ is imperfect if of characteristic $p>0$. It is a Hilbertian ring due to the fact that  its fraction field $K(u_1,\ldots,u_r)$ (with $K={\rm Frac}(R)$) has a product formula, conjoined with \cite[Theorem 4.6]{BDN19}. Furthermore, from Theorem \ref{thm:main_theorem}(c),  $R[u_1,\ldots,u_r]$ satisfies the relative Schinzel Hypothesis. Note further that from Lemma \ref{lem:prelim-proof-R[u]}, the quotients of $R[u_1,\ldots,u_r]$ by prime principal ideals are infinite ($Z\not= \Ff_q[u_1]$ in Theorem \ref{thm-intro2}) and hence assumption (AV2) for the coefficients in $Z[t]$ of $P$ (viewed as a polynomial in $\underline y$) is satisfied. But that assumption (AV2) is exactly assumption (AV3) of Proposition \ref{thm-intro2-general}, which is then satisfied as well. Thus Proposition \ref{thm-intro2-general} applies (with $\underline y=y$ a single variable) and yields the conclusion of Theorem \ref{thm-intro2}.
\end{proof}

\bibliographystyle{alpha}
\bibliography{gcdR1.bib}
\end{document}